\documentclass[a4paper,10pt]{article}
\usepackage[utf8]{inputenc}

\usepackage{amsmath,amssymb,amsfonts,graphicx,amsthm,xcolor}
\usepackage[caption=false]{subfig}

\usepackage[ruled]{algorithm2e}		
\usepackage{myshortcuts}		

\theoremstyle{plain}
\newtheorem{thm}{Theorem}[section]

\newtheorem{obs}[thm]{Observation}
\newtheorem{remark}[thm]{Remark}
\newtheorem{lemma}[thm]{Lemma}
\newtheorem{proposition}[thm]{Proposition}

\def\keywordname{{\bfseries Keywords}}%
\def\keywords#1{\par\addvspace\medskipamount{\rightskip=0pt plus1cm
\def\and{\ifhmode\unskip\nobreak\fi\ $\cdot$
}\noindent\keywordname\enspace\ignorespaces#1\par}}
\def\subclassname{{\bfseries Mathematics Subject Classification
(2000)}\enspace}
\def\subclass#1{\par\addvspace\medskipamount{\rightskip=0pt plus1cm
\def\and{\ifhmode\unskip\nobreak\fi\ $\cdot$
}\noindent\subclassname\ignorespaces#1\par}}

\title{Krylov methods for low-rank commuting \\ generalized Sylvester equations}

\author{Elias Jarlebring \footnotemark[1], Giampaolo Mele
\thanks{Department of Mathematics, KTH Royal Institute of Technology, SeRC swedish e-science research center, Lindstedtsv\"agen 25, SE-100 44 Stockholm, Sweden, email: \{eliasj,gmele,eringh\}@kth.se},
Davide Palitta
\thanks{
Dipartimento di Matematica, Università di Bologna, Piazza di Porta S. Donato, 5, I-40127
Bologna, Italy, email: davide.palitta3@unibo.it
},
Emil Ringh \footnotemark[1]
}

\usepackage[pdftex]{hyperref}	

\begin{document}

\maketitle

\begin{abstract}
  We consider generalizations of the Sylvester matrix equation,  consisting of the sum of a Sylvester operator and a linear operator $\Pi$ with a particular structure. More precisely, the commutator of the matrix coefficients of the operator $\Pi$ and the Sylvester operator coefficients are assumed to be matrices with low rank.  We show (under certain additional conditions) low-rank approximability of this problem, i.e., the solution to this matrix equation can be approximated with a low-rank matrix. Projection methods have successfully been used to solve other matrix equations with low-rank  approximability. We propose a new projection method for this class of matrix equations. The choice of subspace is a crucial ingredient for any projection method for matrix equations. Our method is based on an adaption and extension of the extended Krylov subspace method for Sylvester equations.  A constructive choice of the starting vector/block is derived from the low-rank commutators. We illustrate the 
effectiveness of our method by 
solving large-scale matrix equations arising from applications in control theory and the discretization of PDEs. The advantages of our approach in comparison to other methods are also illustrated.

\subclass{39B42 \and 65F10 \and 58E25 \and 47A46 \and 65F30}

\end{abstract}

\section{Introduction}
Let $\LLL:\RR^{n \times n} \rightarrow \RR^{n \times n}$ denote the \emph{Sylvester operator} associated
with the matrices $A,B\in\RR^{n\times n}$, i.e.,
\begin{equation}  \label{eq:L}
 \LLL(X)	:=AX+XB^T,
\end{equation}
and let $\Pi:\RR^{n \times n} \rightarrow \RR^{n \times n}$ denote the matrix
operator defined by 
\begin{equation} 
 \Pi(X)		:=\sum_{i=1}^m N_i X M_i^T,    \label{eq:Pi}
\end{equation}
where $m\ll n$.  The matrices $A,B$ are assumed to be large and sparse.
Given $C_1,C_2\in \RR^{n \times r}$ with $r \ll n$, 
our paper concerns the problem of computing $X\in\RR^{n \times n}$ such that 
\begin{align} \label{eq:Sylv}
 \LLL(X)+\Pi(X)=C_1C_2^T.
\end{align}
This equation is sometimes (e.g. \cite{DammBenner}) referred to as
the \emph{generalized Sylvester equation}.

Let $[A,B]:=AB-BA$ denote the \emph{commutator} of two matrices.
The structure of the operator $\Pi$ is assumed to be such that
the commutator of the Sylvester coefficients and the coefficients defining the operator $\Pi$ have low rank. In other words, we assume that there exist 
$U_i, \tilde  U_i \in \RR^{n \times s_i}$ and 
$ Q_i, \tilde  Q_i \in \RR^{n \times t_i}$ 
such that $s_i, t_i \ll n$ and the commutators fulfill
\begin{subequations}\label{eq:com}
\begin{eqnarray}
\ [A,N_i] &=& A N_i - N_i A = U_i \tilde U_i^T,   	\\
\ [B,M_i] &=& B M_i - M_i B = Q_i \tilde Q_i^T,    
\end{eqnarray}
\end{subequations}
for $i=1,\dots,m$. 

A recent successful method class for matrix equations defined by large and sparse matrices, are based on projection, typically called \emph{projection methods}  \cite{Simoncini2007, Druskin.Simoncini.11,Benner:2013:low}. We propose a new projection method for \eqref{eq:Sylv} under the low-rank commutation assumption~\eqref{eq:com}. 

Projection methods are typically derived from an assumption on the decay of the singular
values of the solution. More precisely, a necessary condition for the successful application of a projection method is low-rank approximability, i.e., the solution can be approximated by a low-rank matrix. We characterize the low-rank approximability of the solution to~\eqref{eq:Sylv} under the condition that the Sylvester operator $\LLL$ has a low-rank approximability property and that $\rho(\Linv \Pi)<1$. The low-rank approximability theory is presented in Section~\ref{sect:preliminaries}. The function $\rho(\cdot)$ denotes the (operator) spectral radius, i.e.,  
$\rho(\Lop):=\sup\{|\lambda|\, |\, \lambda \in \Lambda(\Lop)\}$.

The choice of the subspace is an important ingredient in any projection method.
We propose a particular choice of projection spaces by identifying certain 
properties of the solution to~\eqref{eq:Sylv}
based on our characterization of low-rank approximability and the low-rank commutation properties~\eqref{eq:com}. More precisely we use an extended Krylov subspace with an appropriate choice of the starting block. We present and analyse an expansion of the framework of extended Krylov subspace method for Sylvester equation (K-PIK) \cite{Simoncini2007,KrylovSylv} to the generalized Sylvester 
equation (Section~\ref{sect:Krylov}).

Linear matrix equations of the form \eqref{eq:Sylv} arise in different applications. For example, the \emph{generalized Lyapunov equation}, which corresponds to the special case where $B=A$, $M_i=N_i$ and $C_1=C_2$, arises in model order reduction of bilinear and stochastic systems, see e.g. \cite{DammBenner,DammDirectADI,Benner:2013:low} and references therein. Many problems arising from the discretization of PDEs can be formulated as generalized Sylvester equations \cite{ringh2016sylvester, Powell2017, Palitta2016}.
Low-rank approximability for matrix equations has been investigated in different settings: for Sylvester equations  \cite{grasedyck2004existenceH, baker2015fast, Grasedyck:2004:Existence}, generalized Lyapunov equations with low-rank correction \cite{Benner:2013:low} and more in general for linear systems with tensor product structure~\cite{kressner2010krylov,Grasedyck:2004:Existence}.
 
The so-called low-rank methods, which projection methods belong to, directly compute a low-rank approximation to the solution of~\eqref{eq:Sylv}. Many algorithms have been developed for the Sylvester equation: projection methods \cite{Simoncini2007, Druskin.Simoncini.11}, low-rank ADI \cite{Benner2009, Benner2014}, sign function method \cite{Baur2008, Baur2006}, Riemannian optimization methods \cite{kressner2016preconditioned,vandereycken2010riemannian} and many more. See the thorough presentation in \cite{Simoncini:2016:Computational}. For large-scale generalized Sylvester equations, fewer numerical methods are available in the literature. Moreover, they are often designed only for solving the generalized Lyapunov equation although they may be adapted to solve the generalized Sylvester equation. In \cite{Benner:2013:low}, the authors propose a bilinear ADI (BilADI) method which naturally extends the low-rank ADI algorithm for standard Lyapunov problems to generalized Lyapunov equations. A non-stationary 
iterative method is derived in \cite{Shank2016}, and in \cite{kressner2015truncated} a greedy low-rank technique is presented. In principle, it is always possible to consider the $n^2\times n^2$ linear system which stems from equation \eqref{eq:Sylv} by Kronecker transformations. There are specific methods for solving linear systems with tensor product structure, see \cite{kressner2015truncated,kressner2016preconditioned,ballani2013projection} and references therein. These problems can also be solved employing one of the many methods for linear systems presented in the literature. In particular, matrix-equation oriented versions of iterative methods for linear systems, together with preconditioning techniques, are present in literature. See, e.g., \cite[Section 5]{Benner:2013:low}, \cite{bouhamidi2008note,kressner2010krylov,li2010numerical}. 
To our knowledge, the low-rank commutativity properties \eqref{eq:com} have not been considered in the literature in the context of methods for matrix equations.

The paper is structured as follows. 
In Section~\ref{sect:preliminaries} we use a Neumann series (cf. \cite{Lancaster1970,Richter:1993:EFFICIENT}) with hypothesis $\rho(\Linv \Pi)<1$ to characterize the low-rank approximability of the solution to \eqref{eq:Sylv}.
In Section~\ref{sect:Krylov} we further characterize approximation properties of the solution to \eqref{eq:Sylv} by exploiting the low-rank commutation feature of the coefficients~\eqref{eq:com}. We use this characterization in the derivation of an efficient projection space.
In Section~\ref{sec:small} we present an efficient procedure for solving small-scale generalized Sylvester equations~\eqref{eq:Sylv}.
Numerical examples that illustrate the effectiveness of our strategy are reported in Section~\ref{sec:num}. 
Our conclusions are given in Section \ref{sec:conc}. 

We use the following notation. The vectorization operator $\vecc: \RR^{n \times n} \rightarrow \RR^{n^2}$ is defined such that $\vecc(A)$ is the vector obtained by stacking the columns of the matrix $A$ on top of one another. We denote by $\| \cdot \|_F$ the Frobenius norm, whereas $\| \cdot \|$ is any submultiplicative matrix norm. For a generic linear and continuous operator $\LLL : \RR^{n \times n} \rightarrow  \RR^{n \times n}$, the induced norm is defined as 
$\| \LLL \| := \inf_{\| A \|=1} \| \LLL(A) \|$. The identity and the zero matrices are respectively denoted by $I$ and $O$. 
We denote by $e_i$ the $i$-th vector of the canonical basis of $\mathbb{R}^n$ while $\otimes$ corresponds to the 
Kronecker product. The matrix obtained by stacking the matrices $A_1, \dots, A_n$ next to each other is denoted by $(A_1, \dots, A_n)$. In conclusion $\vspan(A)$ is the vector space generated by the columns of the matrix $A$ and $\sspan(\A)$ is the vector space generated by the vectors in the set $\A$.

\section{Representation and approximation of the \\ solution} \label{sect:preliminaries}

\subsection{Representation as Neumann series expansion} \label{sect:Neumann}

The following theorem gives sufficient conditions for the existence of a representation of the solution to a generalized Sylvester equation \eqref{eq:Sylv}
as a convergent series. This will be needed for the low-rank approximability
characterization in the following section, as well as in the derivation of a method for small generalized Sylvester equations (further described in Section~\ref{sec:small}).

\begin{thm}[Solution as a Neumann series] \label{thm:representation_solution}
Let $\LLL, \Pi: \RR^{n \times n} \rightarrow \RR^{n \times n}$ be linear operators such that $\LLL$ is invertible and $\rho(\LLL^{-1} \Pi) < 1$ and let $C \in \RR^{n \times n}$. The unique solution of the equation $\LLL(X)+\Pi(X)=C$ can be represented as
 \begin{align} \label{eq:Xseries}
  X = \sum_{j=0}^\infty Y_j,
 \end{align}
 where 
\begin{align} \label{eq:Yj}
\begin{cases}
   Y_0 		&:= \ \ \LLL^{-1} \left( C \right),		\\
   Y_{j+1} 	&:=-\LLL^{-1} \left(\Pi \left( Y_j \right)\right), \qquad \qquad j\geq0. \\
 \end{cases}
 \end{align}
\end{thm}

\begin{proof}
By using the invertibility of $\LLL$ we have $X=(I+\Linv \Pi)^{-1} \Linv(C)$ and 
with the assumption $\rho(\LLL^{-1} \Pi)<1$ we can express the operator $(I+\LLL^{-1} \Pi)^{-1}$ as a convergent Neumann series (for operators as, e.g., in \cite[Example 4.5]{Kato:1995:Pertubation}). In particular, we obtain 
\begin{align*} 
 X =  \sum_{j=0}^\infty (-1)^j \left( \LLL^{-1} \Pi \right)^j \LLL^{-1} \left( C \right).
\end{align*}
The relation \eqref{eq:Xseries} follows by defining 
$Y_j := (-1)^j \left( \LLL^{-1} \Pi \right)^j \LLL^{-1} \left( C \right)$. By induction it follows that the relations~\eqref{eq:Yj} are fulfilled.
\end{proof}

\begin{remark} \label{rmk:XN}
Theorem~\ref{thm:representation_solution} can be used to construct an approximation to the solution of $\LLL(X)+\Pi(X)=C$ by truncating the series \eqref{eq:Xseries} analogous to the general form in \cite[(4.23)]{Kato:1995:Pertubation}.
In particular, let
 \begin{align} \label{eq:XN}
  X^{(\ell)} := \sum_{j=0}^\ell Y_j,
 \end{align}
 where $Y_j$ are given by \eqref{eq:Yj}. The truncation error can be bounded as follows
 \begin{align*} 
  \| X - X^{(\ell)} \| \le \| \Linv(C) \| \frac{\rho(\LLL^{-1}\Pi)^{\ell+1}}{1-\rho(\LLL^{-1}\Pi)}.
 \end{align*}
\end{remark}
If $\LLL$ and $\Pi$ are respectively the operators~\eqref{eq:L} and \eqref{eq:Pi}
that define the generalized Sylvester equation~\eqref{eq:Sylv}, then the truncated Neumann series \eqref{eq:XN} can be efficiently computed for small scale problems. In particular, this approach can be used in the derivation of a numerical method for solving small scale generalized Sylvester equations as illustrated in Section~\ref{sec:small}.

\subsection{Low-rank approximability} \label{sect:low_rank}
We now use the result in the previous section
to show that the solution to \eqref{eq:Sylv} can often
be approximated by a low-rank matrix.
We base the reasoning on low-rank approximability properties
of $\LLL$. Our result requires the explicit use of certain conditions on the spectrum of matrix coefficients of $\LLL$. Under these specific conditions, the solution to a Sylvester equation with low-rank right-hand side can be approximated by a low-rank matrix, see \cite[Section 4.1]{Simoncini:2016:Computational}. In this sense, we can extend several results concerning the low-rank approximability for the solution to the Sylvester equation to the case of generalized Sylvester equations under the assumption $\rho(\Linv\Pi)<1$. 
More precisely, the truncated Neumann series~\eqref{eq:XN} is obtained by summing the solutions to the Sylvester equations~\eqref{eq:Yj}. Note that, under the low-rank approximability assumption of $\LLL$, the right-hand side of the Sylvester equations~\eqref{eq:Yj} is a low-rank matrix since we assume that $C$ is a low-rank matrix and $m \ll n$. We formalize this argument and present a new characterization of the low-rank approximability of the solution to~\eqref{eq:Sylv} by adapting one of the most commonly used low-rank approximability result for
Sylvester equations~\cite{Grasedyck:2004:Existence}. 

We now briefly recall some results presented in \cite{Grasedyck:2004:Existence},
for our purposes. Suppose that the matrix coefficients representing $\LLL$ are such that $\lambda(A)\cup\lambda(B) \subset \CC_-$. Let $M\in\CC^{n \times n}$ be such that $\lambda(M) \subset \CC_-$, then its inverse can be expressed as $M^{-1} = \int_0^\infty \exp(tM) dt$. 
The integral can be approximated with the following quadrature formula 
\begin{align} \label{eq:quadrature_approx}
 M^{-1} = \int_0^\infty \exp(tM) dt \approx \sum_{j=-k}^k w_j \exp(t_j M),
\end{align}
where the weights $w_j$ and nodes $t_j$ are given in \cite[Lemma~5]{Grasedyck:2004:Existence}.
More precisely, we have an explicit formula for the approximation error
\begin{align} \label{eq:error_bound_intgral}
\norm{ \int_0^\infty \exp(tM) dt -  \sum_{j=-k}^k w_j \exp( t_j M) } \le K e^{-\pi \sqrt{k}},
\end{align}
where $K$ is a constant that only depends on the spectrum of $M$.
The solution to the Sylvester equation $\LLL(X)=C$  can be explicitly expressed as
$\vecc(X) = (I \otimes A + B \otimes I)^{-1} \vecc(C)$. The solution to this linear system can be approximated by using \eqref{eq:quadrature_approx} for approximating 
the inverse of $I \otimes A + B \otimes I$.
Let $\Linvk: \RR^{n \times n} \rightarrow \RR^{n \times n}$ be the linear operator such that $\Linvk(C)$ corresponds to the approximation \eqref{eq:quadrature_approx}. More precisely, the operator $\Linvk$ satisfies
\begin{align*}
 \vecc( \Linvk(C) ) =  \sum_{j=-k}^k w_j 
 \left[ \exp( t_j B) \otimes  \exp(  t_j A) \right]
 \vecc(C).
\end{align*}
By using the properties of the Kronecker product, it can be explicitly expressed as
\begin{align}\label{eq:Lkdef_explicit}
 \Linvk(C) =  \sum_{j=-k}^k w_j \exp( t_j A) C  \exp(  t_j B^T).
\end{align}
In terms of operators, the error bound \eqref{eq:error_bound_intgral} is $\|\Linv- \Linvk\| \le K e^{-\pi \sqrt{k}}$. The result of the above discussion is summarized in the following remark, which directly follows from \eqref{eq:Lkdef_explicit} or \cite[Lemma~7]{Grasedyck:2004:Existence}, \cite[Lemma~2]{Benner:2013:low}. 

\begin{remark} \label{rmk:SylvLR}
The solution to the Sylvester equation $\LLL(X)=C$ can be approximated by $\bar X = \Linvk(C)$ where $\| X - \bar X \| \le \| C \| K e^{-\pi \sqrt{k}}$,  $\rank(\bar X) \le (2k+1) r$, $K$ is a constant that depends on the spectrum of $\LLL$ and $r$ is the rank of $C$. 
\end{remark}

The following theorem concerns the low-rank approximability of the solution to~\eqref{eq:Sylv}. More precisely, it provides a generalization of Remark~\ref{rmk:SylvLR} to the case of generalized Sylvester equations by using the Neumann series characterization
in Theorem~\ref{thm:representation_solution}. 

\begin{thm}[Low-rank approximability]\label{thm:low_rank_sol}
Let $\LLL$ be the Sylvester operator \eqref{eq:L}, $\Pi$ the linear operator  \eqref{eq:Pi},  $C_1,C_2 \in \RR^{n \times r}$ and $k$ a positive integer.
Let $X^{(\ell)}$ be the truncated Neumann series~\eqref{eq:XN}.
Then there exists a matrix $\bar X^{(\ell)}$ such that  
 \begin{align}\label{eq:rank:bound} 
 \rank(\bar X^{(\ell)})  \le  (2k+1) r + \sum_{j=1}^\ell (2k+1)^{j+1} m^j r, 
 \end{align}
 and 
 \begin{align} \label{eq:XN_approx_lr}
  \norm{X^{(\ell)} - \bar X^{(\ell)}} \le \bar K e^{-\pi \sqrt{k}},
 \end{align}
where $\bar K$ is a constant that does not depend on $k$ and only depends on $\LLL$ and $\ell$.
 
\end{thm}

\begin{proof}
Let $\LLL_k$ be the operator \eqref{eq:Lkdef_explicit} and consider the sequence 
\begin{align} \label{eq:YiLR}
\begin{cases}
  \bar Y_0 	&:= \ \ \Linvk (C_1 C_2^T),		\\
  \bar Y_{j+1} 	&:=  - \Linvk (\Pi(\bar Y_{j})), \quad j \geq 0.
\end{cases}
\end{align}
Define $\beta:=\| \Linv \Pi \|$ and $\beta_k:=\| \Linvk \Pi \|$. By using Remark~\ref{rmk:SylvLR} we have  
 \begin{align*}
  \| Y_{j+1} - \bar Y_{j+1} \|   & \le	\| \Linv (\Pi (Y_{j})) - \Linv (\Pi (\bar Y_{j})) \| + \| \Linv (\Pi (\bar Y_{j})) - \Linvk (\Pi (\bar Y_{j})) \|		
    \nonumber  \\
  & \le \beta \| Y_{j} - \bar Y_{j} \| + K e^{-\pi \sqrt{k}} \| \Pi \| \| \bar  Y_{j}\|. 
 \end{align*} 
 From the above expression, a simple recursive argument shows that 
  \begin{align} \label{eq:Ydiff}
   \| Y_{j+1} - \bar Y_{j+1} \|  &\le \beta^{j+1}  \| Y_{0} - \bar Y_{0} \| + K e^{-\pi \sqrt{k}} \| \Pi \|  \sum_{t=0}^{j} \beta^{j-t} \| \bar Y_t \|.
 \end{align} 
Using the sub-multiplicativity of the operator norm, it holds that $\| \bar Y_{j} \| = \| \Linvk (\Pi( \bar Y_{j-1} )) \| \le \beta_k \| \bar Y_{j-1} \|$. In particular $\| \bar Y_j \| \le \beta_k^{j} \|\Linvk\| \| C_1C_2^T \|$, and therefore, by using Remark~\ref{rmk:SylvLR}, from \eqref{eq:Ydiff} it follows that 
  \begin{align} \label{eq:Ydiff_2}
\| Y_{j+1} - \bar Y_{j+1} \|   & \le \| C_1 C_2^T \|  K \left[ \beta^{j+1} +  \| \Pi \| \| \Linvk \|\sum_{t=0}^{j} \beta^{j-t} \beta_k^t \right]  e^{-\pi \sqrt{k}}.
 \end{align}
 Since $\Linvk$ converges to $\Linv$, and by using the continuity of the operators, we have that $\|\Linvk\|$ and $\beta_k$ are bounded by a constant independent of $k$.
 Therefore from \eqref{eq:Ydiff_2} it follows that there exists a constant $K_{j+1}$ independent of $k$ such that $\| Y_{j+1} - \bar Y_{j+1} \| \le K_{j+1} e^{-\pi \sqrt{k}}$.
 The relation \eqref{eq:XN_approx_lr} follows by defining $\bar X^{(\ell)} := \sum_{j=0}^\ell \bar{Y}_j$ and observing
  \begin{align*}
  \| X^{(\ell)} - \bar X^{(\ell)} \| & \le 
  \sum_{j=0}^\ell \| Y_j - \bar Y_j \| \le e^{-\pi \sqrt{k}} \sum_{j=0}^\ell K_j = \bar K  e^{-\pi \sqrt{k}},
 \end{align*}
 where $\bar K := \sum_{j=0}^\ell K_j$. The upper-bound~\eqref{eq:rank:bound} follows by Remark~\ref{rmk:SylvLR} iteratively applied to~\eqref{eq:YiLR}.
 \end{proof}

We want to point out that, although Theorem~\ref{thm:low_rank_sol} provides an explicit procedure for constructing an approximation to the solution of~\eqref{eq:Sylv}, we later consider a different class of methods. Theorem~\ref{thm:low_rank_sol} has only theoretical interest and it is used to motivate the employment of low-rank methods in the solution of \eqref{eq:Sylv}. Moreover, in the numerical simulations (Section~\ref{sec:num}), we have observed a  decay in the singular values of the solution to \eqref{eq:Sylv} that it is faster than the one predicted by 
Theorem~\ref{thm:low_rank_sol}.

\section{Structure exploiting Krylov methods}\label{sect:Krylov}

\subsection{Extended Krylov subspace method}
In this section we derive a method for \eqref{eq:Sylv} that belongs to the class called projection methods.
We briefly summarize the adaption of the projection method approach in our setting. 
Projection methods for matrix equations are iterative algorithms based on constructing two sequences of nested subspaces of $\mathbb{R}^n$, i.e., $\mathcal{K}_{k-1}\subset  \mathcal{K}_k$ and $\mathcal{H}_{k-1}\subset \mathcal{H}_k$.
Justified by the low-rank approximability of the solution,
projection methods construct approximations (of the solution to \eqref{eq:Sylv}) of the
form
\begin{equation}
  X_k=\mathcal{V}_k Z_k \mathcal{W}_k^T,\label{eq:Xk_asm}
\end{equation}
where $\mathcal{V}_k$ and $\mathcal{W}_k$ are matrices with orthonormal
columns representing respectively an orthonormal basis of $\mathcal{K}_k$ 
and $\mathcal{H}_k$. Note that low-rank approximability (in the sense illustrated in, e.g., Theorem~\ref{thm:low_rank_sol}) is a necessary condition for the success of an approximation of the type \eqref{eq:Xk_asm}.

The matrix $Z_k$ can be obtained by imposing the Galerkin orthogonality condition, namely the residual 
\begin{align} \label{eq:residual}
\R_k:=AX_k+X_kB^T+
\sum_{i=1}^m
N_iX_kM_i^T-C_1C_2^T,
\end{align}
is such that $\mathcal{V}_k^T \R_k \mathcal{W}_k=0$. 
This condition is equivalent to $Z_k$ satisfying the following small and dense generalized Sylvester equation, usually referred to as the \emph{projected problem,}
\begin{align} \label{eq:proj_problem}
  T_kZ_k+Z_kH_k^T+
  \sum_{i=1}^m
 G_{k,i}Z_kF_{k,i}^T =
 E_{k,1} E_{k,2}^T,
\end{align}
where,
\begin{subequations}
\begin{align}
\label{eq:proj_matrix}
& T_k:=\mathcal{V}_k^TA \mathcal{V}_k, 
& H_k:=\mathcal{W}_k^TB\mathcal{W}_k, \ \
&&E_{k,1}=\mathcal{V}_k^T C_1,
&&E_{k,2}=\mathcal{W}_k^T C_2, 		\\
\label{eq:proj_matrix2}
&G_{k,i}:=\mathcal{V}_k^TN_i \mathcal{V}_k, 
&F_{k,i}:=\mathcal{W}_k^TM_i\mathcal{W}_k,
&& i=1, \dots, m. \ \ \
\end{align}
\end{subequations}
The iterative procedure consists in expanding the spaces $\mathcal{K}_k$ and $\mathcal{H}_k$ until the norm of the residual matrix $\R_k$ \eqref{eq:residual} is sufficiently small.

A projection method is efficient only if the subspaces $\mathcal{K}_k$ and $\mathcal{H}_k$ are selected in a way that the projected matrix \eqref{eq:Xk_asm} is a good low-rank approximation to the solution without the dimensions of the spaces being large.
One of the most popular choices of subspace is the extended
Krylov subspace (although certainly not the only choice \cite{jaimoukha1994krylov,Druskin.Simoncini.11}).
Extended Krylov subspaces form the basis of the method
called Krylov-plus-inverted Krylov (K-PIK) \cite{Simoncini2007,KrylovSylv}.
For our purposes it is natural to define
extended Krylov subspaces with the 
notation of block Krylov subspaces used, e.g.,
in  \cite[Section 6]{Gutknecht:2007:Block}.
Given an invertible matrix $A\in\RRnn$ and $C\in \RR^{n\times r}$, an
extended block Krylov subspace can be defined as the sum of two vector spaces, more precisely 
$\EK_k(A,C):=\K_k(A,C)+\K_{k}(A^{-1},A^{-1}C)$, where 
\begin{align*}
\K_{k}(A,C):= 
\sspan
\left(
\left \{
p(A) C w
\ | \  
\deg(p) \le k
, \
w \in \RR^r
\right \}
\right),
\end{align*}
denotes the block Krylov subspace, $p \in \RR[x]$ is a polynomial, and $\deg(\cdot)$ is the degree function. The extended Krylov subspace method is a projection method where  $\mathcal{K}_k=\EK_k(A,\bar C_1)$, $\mathcal{H}_k=\EK_k(B,\bar C_2)$
and $\bar C_1$, $\bar C_2$ are called the starting blocks, which we will show
how to select in our setting in Sections~\ref{subsec:KrylovAndLowRank} and \ref{sect:low_rank_Krylov}.
The procedure is summarized in Algorithm~\ref{alg:Krylov}  where the matrices $L$ and $R$ are the low-rank factors of \eqref{eq:Xk_asm}, i.e., they are such that $X_k=LR^T$. Notice that in the case of generalized Lyapunov equations the matrices $V_k$ and $W_k$ are equal and Algorithm~\ref{alg:Krylov} can be optimized accordingly.

\begin{algorithm} \label{alg:Krylov}
\small
\caption{Extended Krylov subspace method for generalized Sylvester equations.}
\SetKwInOut{Input}{input}\SetKwInOut{Output}{output}
\Input{Matrix~coeff.:~$ A,B,N_1\dots,N_m,M_1,\dots,M_m \in \RR^{n \times n}$~,~$C_1,C_2 \in \RR^{n \times r}$ 	\\
Starting blocks: $\bar C_1, \bar C_2 \in \RR^{n \times \bar r}$  	\\
Maximum number of iterations: $d$
}
\Output{Low-rank factors: $L, R$}
\BlankLine
\nl Set $V_1=\orth{(\bar C_1, A^{-1} \bar C_1)}$,  
$W_2=\orth{(\bar C_2, B^{-1} \bar C_2)}$, 
$\mathcal{V}_{0}=\mathcal{W}_{0}=\emptyset$
\\
\For{$k = 1,2,\dots, d$}{
\nl $\mathcal{V}_{k}=(\mathcal{V}_{k-1}, V_{k})$ and
$\mathcal{W}_{k}=(\mathcal{W}_{k-1}, W_{k})$	\\
\nl Compute $T_k, H_k, E_{k,1}, E_{k,2}, G_{k,i}, F_{k,i}$ according to \eqref{eq:proj_matrix}-\eqref{eq:proj_matrix2}
\label{step:proj_mat}
\\
\nl Solve the \emph{projected problem} \eqref{eq:proj_problem} \label{step:proj_problem} \\
\nl Compute $\| \R_k\|_F$ according to \eqref{eq:residual_norm} \\
\If{$\| \R_k\|_F \le \tol$}{Break} 
\begin{tabular}{lcll}
\hspace{-0.3cm} \nl Set $V_k^{(1)}$: &first $\bar r$ columns of $V_k$; & Set $V_k^{(2)}$: &last $\bar r$ columns of $V_k$	
\\
\hspace{-0.3cm} \nl Set $W_k^{(1)}$: & \ first $\bar r$ columns of $W_k$; & Set $W_k^{(2)}$: &last $\bar r$ columns of $W_k$	\\
\end{tabular}
\nl $V'_{k+1}=(AV_k^{(1)}, A^{-1}V_k^{(2)})$ and $W'_{k+1}=(BW_k^{(1)}, B^{-1} W_k^{(2)})$ \\
\begin{tabular}{lcllll}
\hspace{-0.3cm} \nl 
$\widehat V_{k+1}$ & $\leftarrow$ & \mbox{block-orthogonalize} & $V'_{k+1}$ &w.r.t.& $\mathcal{V}_k$	
\label{step:genKa}
\\
\hspace{-0.3cm} \nl  $\widehat W_{k+1}$ & $\leftarrow$ & \mbox{block-orthogonalize} & $W'_{k+1}$ &w.r.t.& $\mathcal{W}_k$	
\end{tabular}
\nl 
$V_{k+1}=\orth{\widehat V_{k+1}}$ and 
$W_{k+1}=\orth{\widehat W_{k+1}}$ 
\label{step:genKb}
}
\nl Compute the decomposition $Z_k=\widehat L \widehat R^T$		\\
\nl Return $L=\mathcal{V}_k \widehat L$ and $R=\mathcal{W}_k \widehat R$
\end{algorithm}

\begin{remark} \label{rmk:S1S2}
The output of Algorithm~\ref{alg:Krylov} represents the factorization $X_k=L R^T$. Under the condition that $\|\R_k\|$ is small, $X_k$ is an approximation of the solution of the generalized Sylvester equation~\eqref{eq:Sylv} such that $\rank(X_k) \le 2 \bar r k$. 
By construction $\vspan(L) \subseteq \EK_k(A,\bar C_1)$ and $\vspan(R) \subseteq \EK_k(B,\bar C_2)$. For the case of the Sylvester equation, $m=0$, Algorithm \ref{alg:Krylov} can be employed with the natural choice of the starting blocks $\bar C_1=C_1$ and $\bar C_2=C_2$, as it has been shown, e.g., in \cite{Simoncini2007,KrylovSylv}.
\end{remark}

A breakdown in Algorithm~\ref{alg:Krylov} may occur in two situations. During the generation of the basis of the extended
Krylov subspaces, (numerical) loss of orthogonality may occur in Steps~\ref{step:genKa}-\ref{step:genKb}. 
This issue is present already for the Sylvester equation \cite{Simoncini2007,KrylovSylv} and we refer to \cite{Gutknecht:2007:Block} for a presentation of safeguard strategies that may mitigate the problem.  
We assume that the bases $\mathcal{V}_k$ and $\mathcal{W}_k$ have full rank.
The other situation where a breakdown may occur is in Step~\ref{step:proj_problem}.
It may happen that the projected problem \eqref{eq:proj_problem} is not solvable.
For the  Sylvester equation the solvability of the projected problem is guaranteed by the condition that the field of values of $A$ and $B$ are disjoint \cite[Section 4.4.1]{Simoncini:2016:Computational}.
We extend this result, which provides a way to verify the applicability of the
method (without carrying out the method). As illustrated in the following proposition,
for the generalized Sylvester equation we need an additional condition.
Instead of using the field of values, it is natural to phrase this condition
in terms of the ratio field of values (defined in, e.g., \cite{Einstein2011}).

\begin{proposition} 
Consider the generalized Sylvester equation \eqref{eq:Sylv} and assume that the field of values of $A$ and $B$ are disjoint, and that the ratio field of values of $\sum_{i=1}^m M_i \otimes N_i$ and $B\otimes I+I\otimes A$, i.e.,
{\small
\begin{align*}
 R
 \left(
  \sum_{i=1}^m
 M_i\otimes N_i,B\otimes I+I\otimes A
 \right)
 :=\left\{
 \frac{y^H\left(
 \sum_{i=1}^m M_i\otimes N_i
 \right)y}{y^H\left(B\otimes I+I\otimes A\right)y} \; \bigg| \; y\in\mathbb{C}^{n^2} \setminus 
 \left \{ 0 \right \} \right\},
  \end{align*}
  }is strictly contained in the open unit disk. 
 Then the projected problem~\eqref{eq:proj_problem} has a 
 unique solution.
\end{proposition}

\begin{proof} 
Let $\mathcal{L}_{proj}(Z):=T_kZ+ZH_k^T$ and 
 $\Pi_{proj}(Z):=
  \sum_{i=1}^m
 G_{k,i}ZF_{k,i}^T$. The projected problem \eqref{eq:proj_problem} is equivalently written as 
 $\mathcal{L}_{proj}(Z_k)+\Pi_{proj}(Z_k)=E_{k,1} E_{k,2}^T$.
 Since $A$ and $B$ have disjoint fields of values, $\mathcal{L}_{proj}$ is invertible \cite[Section 4.4.1]{Simoncini:2016:Computational}.
 From Theorem~\ref{thm:representation_solution} we know that it exists a unique solution $Z_k$ to \eqref{eq:proj_problem} 
if $\rho\left(\mathcal{L}_{proj}^{-1}\Pi_{proj}\right)<1$. 
This condition is equivalent to $|\lambda|<1$, where 
$(\lambda, v)\in \CC \times \mathbb{C}^{(kr)^2} \setminus \left \{ 0 \right \}$ 
is an eigenpair of the following generalized eigenvalue problem
 \begin{equation}\label{gen_eigen.1}
 \left(
 \sum_{i=1}^m
 F_{k,i}\otimes G_{k,i}
 \right)v=\lambda (H_k\otimes I+I\otimes T_k)v.
 \end{equation}
 Using the properties of the Kronecker product, equation \eqref{gen_eigen.1} can be written as
 {\small
 $$
 \sum_{i=1}^m
 (W_k^T\otimes V_k^T)\left(M_i\otimes N_i\right)\left(W_k\otimes V_k\right)v=\lambda
(W_k^T\otimes V_k^T)\left(B\otimes I+I\otimes A\right)\left(W_k\otimes V_k\right)v.$$
}
 By multiplying the above equation from the left with $v^H$ we have that 
 \begin{align*}
 |\lambda|
 &=\left|\frac{x^H\left(
 \sum_{i=1}^m
 M_i\otimes N_i\right)x}{x^H\left(B\otimes I+I\otimes A\right)x}\right|, \quad
 x:=\left(W_k\otimes V_k\right)v. 
 \end{align*}
 By using that $ R\left(
  \sum_{i=1}^m
 M_i\otimes N_i,B\otimes I+I\otimes A
 \right)$ is strictly contained in the unit circle we conclude that $|\lambda|<1$. 
\end{proof}

\begin{obs}
The computation of the matrices $T_k$, $H_k$ (Step~\ref{step:proj_mat}) and the orthogonalization of the new blocks $V_{k+1}, W_{k+1}$ (Steps~\ref{step:genKa}-\ref{step:genKb}) can be efficiently performed as in \cite[Section 3]{Simoncini2007} where a modified Gram-Schmidt method is employed in the orthogonalization. The matrices $G_{k,i}$ and $F_{k,i}$ (Step~\ref{step:proj_mat}) can be computed by extending the matrices $G_{k-1,i}$ and $F_{k-1,i}$ with a block-column and a block-row.   
Moreover, the matrix $X_k$ is never explicitly formed. In particular, the Frobenius norm of the residual \eqref{eq:residual} can be computed as
\begin{align} \label{eq:residual_norm}
 \|\R_k\|_F^2
 & = \|\tau_{k+1}(e_{k} \otimes I_{2r} )^TZ_k\|_F^2+
 \|Z_k(e_{k} \otimes I_{2r} )^T h_{k+1}^T\|_F^2.
\end{align}
This follows by replacing in \eqref{eq:residual} the following 
Arnoldi-like relations~\cite[equation (4)]{knizhnerman2010new}
 $$
 A \mathcal{V}_k=\mathcal{V}_kT_k+V_{k+1}\tau_{k+1}(e_{k} \otimes I_{2r} )^T, \quad 
 B \mathcal{W}_k=\mathcal{W}_k H_k+W_{k+1}h_{k+1}(e_{k} \otimes I_{2r} )^T.
 $$
\end{obs}

\subsection{Krylov subspace and low-rank commuting matrices} \label{subsec:KrylovAndLowRank}

The starting blocks $\bar C_1$ and $\bar C_2$ in 
Algorithm~\ref{alg:Krylov} need to be selected such that the generated
subspaces have good approximation properties. We now present an appropriate way to select these matrices by using certain approximation properties of the solution to~\eqref{eq:Sylv}, under the low-rank commutation property \eqref{eq:com}.

We first need a technical result which shows that if the commutator of two matrices has low rank, then the corresponding commutator, where one matrix is taken to a given power,  has also low rank. The rank increases with the power of the matrix.

\begin{lemma}\label{lem:comm_main_prop}
  Suppose $A$ and $N$ are matrices such that $[A,N]=U\tilde U^T$. Then,
  \[
 [A^{j},N] 
  = 
  \sum_{k=0}^{j-1} 
  A^k U \tilde U^T A^{j-k-1}. 
  \]  
  \end{lemma} 
\begin{proof} 
The proof is by induction. The basis of induction is trivially verified for $j=1$. Assume that the claim is valid for $j$, then the induction step follows by observing that
\begin{align*}
[A^{j+1}, N]
= A^{j+1}N - NA^{j+1}
= A^{j}U\tilde U^T + (A^{j}N-NA^{j})A,
\end{align*}
and applying the induction hypothesis on $A^{j}N-NA^{j}$.
\end{proof}
As pointed out in Remark~\ref{rmk:S1S2}, $C_1$ and $C_2$ are natural starting blocks
for the Sylvester equation. If we apply this result to the sequence
of Sylvester equations in Theorem~\ref{thm:representation_solution}, with $\LLL$ and $\Pi$ defined as 
\eqref{eq:L}-\eqref{eq:Pi}, we obtain
subspaces with a particular structure. 
For example, the approximation $L_0R_0^T$ to $Y_0$ provided by Algorithm~\ref{alg:Krylov} is such that $\vspan(L_0)\subseteq\EK_{k}(A,C_1)$ and $\vspan(R_0)\subseteq\EK_{k}(B,C_2)$. Since $Y_0$ is contained in the right-hand side of the definition of $Y_1$, in order to compute an approximation of $Y_1$, we should consider the subspaces $N_i\,\cdot\nobreak\,\EK_{k}(A,C_1)$ and $M_i\,\cdot\,\EK_{k}(B,C_2)$ for $i=1,\ldots,m$. By using the low-rank commutation property \eqref{eq:com} such subspaces can be characterized by the following result.

\begin{thm} \label{thm:Krylov}
Assume that $A \in \RRnn$ is nonsingular and let $N \in \RRnn$ such that 
$[A,N]=U \tilde U^T$ with $U, \tilde U \in \RR^{n \times s}$. 
Let $C \in \RR^{n \times r}$, then
\[
N\,\cdot\,\EK_{k}(A,C) \subseteq \EK_{k}(A,(NC,U)).
\]
\end{thm}

\begin{proof}
Let $N p(A) C w+N q(A^{-1}) C v$ be a generator of $N \cdot \EK_{k}(A,C)$, where  
$p(x)=\sum_{j=0}^k \alpha_j x^j$. Then, with a direct usage of Lemma~\ref{lem:comm_main_prop}, the vector 
$N p(A) C w$ can be expressed as an element of $\EK_{k}(A,(NC,U))$ in the following way
\begin{align*}
   N p(A) C w 
   &=   N \sum_{j=0}^k \alpha_j A^j C w 
   =	p(A)N Cw - \sum_{j=0}^k \sum_{\ell=0}^{j-1}\alpha_j A^\ell U \left(  \tilde U^T A^{j-1-\ell}  Cw \right).  
  \end{align*}
We can show that $N q(A^{-1}) C v$ belongs to the subspace $\EK_{k}(A,(NC,U))$ with the same procedure and by using that $[A^{-1},N] =  -(A^{-1} U) ( A^{-T} \tilde U)^T$.\end{proof}

In order to ease the notation and improve conciseness of the results that follow, 
we introduce the following multivariate generalization of the Krylov subspace for more matrices 
\begin{align*}
 \GK_d(N_1, \dots, N_m; U) := 
 \sspan 
 \left \{
 p(N_1, \dots, N_m)Uz \middle| 
 \deg(p) \le d, z \in \RR^r
 \right \},
\end{align*}
where $U \in \RR^{n \times r}$ and $p$ is 
a non-commutative multivariate polynomial in the free algebra 
$\RR<x_1, \dots, x_N>$ (in the sense of \cite[Chapter 10]{berstel2011noncommutative}).

\begin{obs}\label{obs:GK}
 Observe that $\GK_d(N_1, \dots, N_m; U)$ 
 is the space generated by the columns of the matrices obtained 
 multiplying (in any order) $s \le d$ matrices $N_i$ and the matrix $U$.
 In particular this space can be equivalently characterized as 
\begin{align*}
 \GK_d(N_1, \dots, N_m; U) = 
 \sspan 
 \left \{
 N_{i_1} \cdots N_{i_s} Uz \middle| 
 1 \le i_j \le m, 
 0 \le s \le d,
 z \in \RR^r
 \right \}.
\end{align*}
This definition generalizes the definition of the standard block Krylov subspace in 
the sense that 
$\GK_d(N; U) = \K_d(N,U)$.
\end{obs}

The solution to the generalized Sylvester equation~\eqref{eq:Sylv} can be approximated  by constructing an approximation of $X^{(\ell)}$. In particular, by subsequentially computing low-rank approximations to the Sylvester equations \eqref{eq:Yj}.
In the following theorem we illustrate some properties that this approximation of $X^{(\ell)}$ fulfills. In order to state the theorem we need the result of the application of the extended Krylov method to the (standard) Sylvester equations of the form
\begin{subequations}
\begin{align}
\label{eq:YY0}
  A \Y + \Y B^T	& = C_1 C_2^T,				\\
\label{eq:YYj}
  A \Y + \Y B^T	& = -\sum_{i=1}^m (N_i L_j) (M_i R_j)^T,
\end{align}
\end{subequations}
as described in \cite{Simoncini2007,KrylovSylv}. As already stated in Remark~\ref{rmk:S1S2}, this is identical to applying Algorithm~\ref{alg:Krylov} with $m=0$.

\begin{thm} \label{thm:Space}
Consider the generalized Sylvester equation \eqref{eq:Sylv}, 
with coefficients commuting according to \eqref{eq:com}. 
Let $\tilde Y_0=L_0 R_0^T$ be the result of Algorithm~\ref{alg:Krylov} applied to the (standard) Sylvester equation~\eqref{eq:YY0} with starting blocks $\bar C_1=C_1$ and $\bar C_2=C_2$. 
Moreover, for $j=0, \dots, \ell-1$, let $\tilde Y_{j+1}=L_{j+1} R_{j+1}^T$ be the result of Algorithm~\ref{alg:Krylov} applied to the Sylvester equation~\eqref{eq:YYj} with starting blocks $\bar C_1=(N_1 L_j, \dots, N_m L_j)$ and $\bar C_2=(M_1 R_j, \dots, M_m R_j)$. Let $\tilde X^{(\ell)}$ be the approximation of the truncated Neumann series \eqref{eq:XN} given by 
$$\tilde X^{(\ell)} := \sum_{j=0}^\ell \tilde Y_j.$$ 
Then, there exist matrices $L,R,\hat C_1^{(\ell)},\hat C_2^{(\ell)}$ such that
$\vspan(L) \subseteq \EK_{(\ell+1)d}( A, \hat C_1^{(\ell)} )$ and 
$\vspan(R) \subseteq \EK_{(\ell+1)d}( B, \hat C_2^{(\ell)} )$ and
\begin{align*}
\tilde X^{(\ell)} = L R^T,
\end{align*}
where
\begin{subequations} \label{eq:C1C2N}
\begin{align}
\label{eq:C1N}
\vspan(\hat C_1^{(\ell)}) \subseteq & \ \GK_{\ell}(N_1, \dots, N_m; C_1) + 
\GK_{\ell-1}(N_1, \dots, N_m; U),	\\
\label{eq:C2N}
\vspan(\hat C_2^{(\ell)}) \subseteq &  \ \GK_{\ell}(M_1, \dots, M_m; C_2) + 
\GK_{\ell-1}(M_1, \dots, M_m; Q),	
\end{align}
\end{subequations}
and 
$U := (U_1,\dots,U_m)$, $Q := (Q_1,\dots,Q_m)$.
\end{thm}

\begin{proof}
We start proving that for $j=0, \dots, \ell$, there exists a matrix $S_j$ such that $\vspan(L_j) \subseteq \EK_{(j+1)d}(A,S_j)$ and
{ 
\small
\medmuskip=1mu
\thinmuskip=1.3mu
\thickmuskip=2mu
\begin{align} \label{eq:induction} 
\hspace{-0.6cm} \nonumber
&\vspan (S_j) \subseteq	\\
& \sspan
 \left \{
 \left( \prod_{i=1}^{j} N_{j_i} \hspace{-0.085cm} \right) C_1 \hspace{-0.05cm} w  + 
 p (N_1, \dots, N_m) U \hspace{-0.05cm} z
  \middle| w \in \RR^r, z\in\RR^{s},
 1 \le j_i \le m, \deg (p) \le j-1 
  \right \},
\end{align}}where $s=\sum_{i=1}^m s_i$ 
and $s_i$ denotes the number of columns of $U_i$. We prove this claim by induction. 
The basis of induction is trivially verified with $S_0:=C_1$ and using Remark~\ref{rmk:S1S2}. We now assume that the claim is valid for $j$ and we perform the induction step. Remark~\ref{rmk:S1S2} implies that $\vspan( L_{j+1} ) \subseteq \EK_d(A,(N_1 L_j,$ $ \dots, N_m L_j) )$. From Theorem~\ref{thm:Krylov} and the induction hypothesis we have that $\vspan( N_i L_j ) \subseteq \EK_{(j+1)d} (A,(N_i, S_j U_i))$ for any $i=1, \dots, m$. Therefore we have that $ \vspan( L_{j+1} ) \subseteq \EK_{(j+2)d} (A, (N_1 S_j, \dots, N_m S_j, U))$. 
We define $S_{j+1}:=(N S_j, \dots, N_m S_j, U)$ which concludes the induction.

From \eqref{eq:induction} we now obtain the relation
\begin{align*}
\vspan((S_1, \dots, S_j)) \subseteq 
\GK_{j}(N_1, \dots, N_m; C_1) + 
\GK_{j-1}(N_1, \dots, N_m; U),	
\end{align*}
that directly implies \eqref{eq:C1N} by setting $\hat C_1^{(\ell)}:=(S_1, \dots, S_\ell)$.
Equation \eqref{eq:C2N} follows from completely analogous reasoning. 
The final conclusion follows by defining $L:=(L_0,\dots, L_\ell)$ and $R:=(R_0,\dots,R_\ell)$.
\end{proof} 

The main message of the previous theorem can be summarized as follows. 
The low-rank factors of the approximation of $X^{(\ell)}$ \eqref{eq:XN} obtained by solving the Sylvester equations~\eqref{eq:Yj} with K-PIK \cite{Simoncini2007,KrylovSylv} (that it is equivalent to  Algorithm~\ref{alg:Krylov} as discussed in Remark~\ref{rmk:S1S2}), are contained in an extended Krylov subspace with a specific choice of the starting blocks. In particular the starting blocks are selected as $\bar C_1 = \hat C_1^{(\ell)}$, $\bar C_2 = \hat C_2^{(\ell)}$ where $\hat C_1^{(\ell)}$ and $\hat C_2^{(\ell)}$ fulfill \eqref{eq:C1N}-\eqref{eq:C2N}. 
Therefore Algorithm~\ref{alg:Krylov} can be used directly to the generalized Sylvester equation~\eqref{eq:Sylv} with this choice of the starting blocks. It is computationally more attractive to use Algorithm~\ref{alg:Krylov} directly on the generalized Sylvester equation~\eqref{eq:Sylv} if the starting blocks are low-rank matrices.
A practical procedure that generates starting blocks that fulfill~\eqref{eq:C1C2N} consists in selecting $\bar C_1$ and $\bar C_2$ such that their columns are respectively a basis of the subspaces $\GK_{\ell}(N_1, \dots, N_m; C_1)$ $+$ $ \GK_{\ell-1}(N_1, \dots, N_m; U)$ and $\GK_{\ell}(M_1, \dots, M_m; C_2) + \GK_{\ell-1}(M_1, \dots, M_m; Q)$. 
A basis of such spaces can be computed by using  Observation~\ref{obs:GK}. For example a basis of $\GK_{2}(N_1, N_2; U)$ is given by the columns of the matrix 
\begin{align*}
(U, \ N_1 U, \ N_2 U, \ N_1 N_2 U, \ N_2 N_1 U, \ N_1^2 U, \ N_2^2 U) .
\end{align*}
Observe that this approach can take advantage of many different features of the original generalized Sylvester equation~\eqref{eq:Sylv}. In certain cases the dimension of the subspaces $\GK_{\ell}$ is bounded for all the $\ell$. This condition is satisfied, e.g., if the matrix coefficients $N_i$, $M_i$ are nilpotent/idempotent or in general if they have  low degree minimal polynomials. Therefore, it is possible to select the starting blocks such that Algorithm~\ref{alg:Krylov} provides an approximation of $X^{(\ell)}$ for all $\ell$, i.e., the full series~\eqref{eq:Xseries} is approximated. These situations naturally appear in applications, see the numerical example in Section~\ref{sect:LaplaceChi}.

\subsection{Krylov subspace method and low-rank matrices} \label{sect:low_rank_Krylov}
Our numerical method can be improved for the
following special case. We now consider a generalized Sylvester equation~\eqref{eq:Sylv} where $N_i = \U_i \tilde \U_i^T$ and $M_i = \Q_i \tilde \Q_i^T$ are low-rank matrices. Obviously, the commutators $[A,N_i]$ and $[B,M_i]$ also have low rank and the theory and the procedure presented in the previous section cover this case. However, the solution to~\eqref{eq:Sylv} can be further characterized and an efficient (and different) choice of the starting blocks $\bar C_1, \bar C_2$ can be derived. The assumption $\rho(\Lop^{-1}\Pi)<1$ is no longer needed in order to justify the low-rank approximability. This property can be illustrated with a Sherman-Morrison-Woodbury formula as proposed in \cite{Benner:2013:low}.
The following proposition shows that, the generalized Sylvester equation~\eqref{eq:Sylv} can be implicitly written as a Sylvester equation with right-hand side involving the matrices $\U_i$ and $\Q_i$ for $i=1,\dots,m$. By using Remark~\ref{rmk:S1S2} this leads to the natural choice of the starting blocks  $\bar C_1=(C_1, \U_1, \dots, \U_m)$ and 
$\bar C_2=(C_2, \Q_1, \dots, \Q_m)$.

\begin{proposition}\label{prop:Krylov_low_rank}
Consider the generalized Sylvester equation \eqref{eq:Sylv}, 
assume that 
$N_i=\U_i \tilde \U_i^T$ and 
$M_i = \Q_i \tilde \Q_i^T$ such that 
$\U_i, \tilde \U_i \in \RR^{n \times s_i}$ and 
$\Q_i, \tilde \Q_i \in \RR^{n \times t_i}$. 
Then there exist $\alpha_i \in \RR$ for $i=1,\dots,m$ such that 
\begin{align*}
 A X + X B^T = C_1 C_2^T - \sum_{i=1}^m \alpha_i \U_i \Q_i^T  
\end{align*}

\end{proposition}
\begin{proof}
The proof follows by \cite[Theorem 4.1]{ringh2016sylvester} setting $E_i:= \U_i \Q_i^T$.
\end{proof} 

\subsection{Solving the projected problem}\label{sec:small}
In order to apply Algorithm~\ref{alg:Krylov} we need
to solve the projected problem in Step~\ref{step:proj_problem}.
The projected problem has to be solved in every
iteration and efficiency is therefore required in practice.
For completeness we now derive a procedure to solve the projected problem based 
on the Neumann series expansion derived in Section~\ref{sect:Neumann},
although this is certainly not the only option. 
The derivation is based on the following observations. 
The projected problem is a small generalized Sylvester equation~\eqref{eq:Sylv}, and
 the computation of $X^{(\ell)}$ in \eqref{eq:XN} requires solving 
$\ell+1$ Sylvester equations~\eqref{eq:Yj}. Since the Sylvester equations \eqref{eq:Yj} are defined by the same coefficients, 
they can be simultaneously reduced to triangular form 

\begin{subequations}\label{eq:Silvester_triang}
\begin{align}
&U_A\widetilde Y_{0}+\widetilde Y_0U_B^T=\widetilde C_1 \widetilde C_2^T, \\
&U_A\widetilde Y_{j+1}+\widetilde Y_{j+1}U_B^T=
-\sum_{i=1}^m
\widetilde N_i \widetilde Y_{j}
\widetilde M_i^T ,
&\quad j=0,\ldots,\ell-1,
\end{align}
\end{subequations}
where we have defined 
\begin{align} \label{eq:triang_Sylv_coeff}
&\widetilde C_1 := Q_A^T C_1, 
& \widetilde C_2:= Q_B^T C_2,	
&&\widetilde N_i:=Q_A^TN_iQ_A, 
&& \widetilde M_i^T:=Q_B^TM_i^TQ_B,
\end{align}
and $A=Q_AU_AQ_A^T$ and $B=Q_BU_BQ_B^T$ denote the Schur decompositions.
The Sylvester equations~\eqref{eq:Silvester_triang} with triangular coefficients can be efficiently solved with backward substitution as in the Bartels-Stewart algorithm~\cite{Bartels1972} and it holds that
$X^{(\ell)}=Q_A \left(\sum_{j=0}^\ell \widetilde Y_j\right)Q_B^T$. 
The Frobenius norm of the residual  
$\R^{(\ell)}:=AX^{(\ell)}+X^{(\ell)} B^T+\sum_{i=1}^mN_iX^{(\ell)}M_i^T-C_1C_2^T$
can be computed without explicitly constructing $X^{(\ell)}$
as follows
 \begin{equation}\label{exact_residualnorm}
  \|\R^{(\ell)}\|_F=
  \left\|\sum_{i=1}^m \widetilde N_i \widetilde Y_{\ell} \widetilde M_i^T\right\|_F.
 \end{equation}
The previous relation follows by simply using the properties of the Frobenius norm (invariance under orthogonal transformations) and the relations~\eqref{eq:Silvester_triang}. 

In conclusion, the following iterative procedure can be used to approximate the solution to \eqref{eq:Sylv}: the matrices 
\eqref{eq:triang_Sylv_coeff} are precomputed, then the Sylvester equations in triangular form~\eqref{eq:Silvester_triang} are solved until the residual of the Neumann series \eqref{exact_residualnorm}
is sufficiently small. The approximation $X^{(\ell)}$ is not computed during the iteration,
but only constructed after the iteration has completed.
The procedure is summarized in Algorithm~\ref{alg:small_size_problems}.

\begin{algorithm}
\label{alg:small_size_problems}
\small
\caption{
Neumann series approach for \eqref{eq:Sylv}.}
\SetKwInOut{Input}{input}\SetKwInOut{Output}{output}
\Input{Matrix coefficients: $A,B,N_1\dots,N_m,M_1,\dots,M_m,C_1,C_2$}
\Output{Truncated Neumann series $X^{(\ell)}$}
\BlankLine
\nl Compute the Schur decompositions $A=Q_AU_AQ_A^T,$ $B=Q_BU_BQ_B^T$\\
\nl Compute $\widetilde C_1$, $\widetilde C_2$, $\widetilde N_i$ $\widetilde M_i$ for all $i=1,\ldots,m$
according to \eqref{eq:triang_Sylv_coeff} \\
\nl Solve $U_A\widetilde Y_{0}+\widetilde Y_0U_B^T=\widetilde C_1 \widetilde C_2^T$ 
and set $\widetilde X =\widetilde Y_{0}$\\
\For{$j = 0,1,\dots $ till convergence}{
\nl Solve $ U_A\widetilde Y_{j+1}+\widetilde Y_{j+1}U_B^T=-\sum_{i=1}^m\widetilde N_i\widetilde Y_{j}\widetilde M_i^T$ and set $\widetilde X =\widetilde X + \widetilde Y_{j+1}$\\
  \nl Compute $\|\R^{(j+1)}\|_F=\|\sum_{i=1}^m\widetilde N_i \widetilde Y_{j+1} \widetilde M_i^T\|_F$\\
\If{$\| \R^{(j+1)}\|_F \le \tol$}{
\nl Set $\ell = j+1$ \\
\nl Break
} 
}
  \nl Return $X^{(\ell)}=Q_A \widetilde X Q_B^T$\\
\end{algorithm}

\section{Numerical examples}\label{sec:num}
We now illustrate our approach with several examples.
In the first two examples, we compare our approach with two different methods for generalized Lyapunov equations:
BilADI \cite{Benner:2013:low} and GLEK \cite{Shank2016}. As expected,
the results are generally in favor of our approach, 
since the other methods are less specialized to the specific structure,
although they have a wider applicable problem domain.
Two variants of BilADI are considered.
In the first variant we select the Wachspress shifts, see e.g., \cite{Wachspress2013}, computed with the software available on Saak's web page\footnote{https://www2.mpi-magdeburg.mpg.de/mpcsc/mitarbeiter/saak/Software/adipars.php}.
In the second variant $\mathcal{H}_2$-optimal shifts \cite{Benner2012} are used. 
The GLEK code is available at the web page of  Simoncini\footnote{http://www.dm.unibo.it/\textasciitilde simoncin/software.html}.
This algorithm requires  fine-tune of several thresholds. We selected {\tt tol\_inexact}$=10^{-2}$ while the default setting is used for all the other thresholds.
The implementation of our approach is based on the modification of K-PIK \cite{Simoncini2007,KrylovSylv} for generalized Sylvester equation as described in Algorithm~\ref{alg:Krylov}. The projected problems,
computed in Step~\ref{step:proj_problem}, are solved with the procedure described in the Section~\ref{sec:small}. A MATLAB implementation of Algorithm~\ref{alg:Krylov} is available online\footnote{http://www.dm.unibo.it/\textasciitilde davide.palitta3}.

In all the methods that we test, the stopping criterion is based on the relative residual norm and the algorithms are stopped when it reaches $\mathtt{tol}=10^{-6}$.
We compare: number of iterations, memory requirements, rank of the computed approximation, number of linear solves (involving the matrices $A$ and $B$ eventually shifted) and total execution CPU-times. 

As memory requirement (denoted Mem. in the following tables) we consider the number of vectors of length $n$ stored during the solution process. In particular, for Algorithm~\ref{alg:Krylov} it consists of the dimension of the approximation space. In GLEK, a sequence of extended Krylov subspaces is generated and the memory requirement corresponds to the dimension of the largest space in the sequence. For the bilinear ADI approach the memory requirement consists of the number of columns of the low-rank factor of the solution. For GLEK, we just report the number of outer iterations. The CPU–times 
reported for BilADI do no take into account the time for the shift computations.
All results were obtained with MATLAB R2015a on a computer with two 2~GHz processors and 128~GB of RAM. 

\subsection{A multiple input multiple output system (MIMO)}
The time invariant multi-input and multi-output (MIMO) bilinear system described in \cite[Example 2]{Lin2009} yields the following generalized Lyapunov equation
\begin{equation}\label{eq.MIMO}
 AX+XA^T+\gamma^2\sum_{i=1}^2N_iXN_i^T=CC^T,
\end{equation}
where $\gamma\in\mathbb{R}$, $\gamma>0$, 
$A=\mbox{tridiag}(2,-5,2)$, $N_1=\mbox{tridiag}(3,0,-3)$ and $N_2=-N_1+I$.
We consider $C\in\mathbb{R}^{n\times 2}$ being a normalized random matrix. In the context of bilinear systems, the solution to \eqref{eq.MIMO}, referred to as \emph{Gramian}, is used for computing energy estimates of the reachability of the states. The number $\gamma$ is a scaling parameter selected in order to ensure the solvability of the problem~\eqref{eq.MIMO} and the positive definiteness of the solution, namely $\rho(\mathcal{L}^{-1}\Pi)<1$. This parameter corresponds to rescaling  the input of the underlying problem with a possible reduction in the region where energy estimates hold. Therefore, it is preferable not to employ very small values of $\gamma$. See \cite{DammBenner} for detailed discussions. 

For this problem the commutators have low rank, more precisely $[A,N_1]=-[A,N_2] = U \tilde U^T$, with $U= 2\sqrt{3} (e_1,e_n)$ and $\tilde U= 2\sqrt{3} (e_1,-e_n)$. As proposed in Section~\ref{subsec:KrylovAndLowRank} we use Algorithm~\ref{alg:Krylov} with starting blocks $\bar C_1= \bar C_2=(C,N_1C,U)$ since $\vspan (C_1^{(1)}) =\vspan((C,N_1C,N_2C,U))= \vspan \left(C,N_1C,U \right)$. Table~\ref{Ex.2_Tab.1} illustrates the performances of our approach and the other low-rank methods, GLEK and the BilADI, as $\gamma$ varies.
\begin{table}[!ht]
\centering
{\small
 \begin{tabular}{r|r|r|r|r|r|r}
  & $\gamma$ & Its. & Mem.
  & rank($X$) & Lin. solves & CPU time \\
  \hline 
  \hline
  BilADI (4 Wach.) & 1/6 & 10 & 55 & 55 & 320 & 51.26 \\
  BilADI (8 $\mathcal{H}_2$-opt.) & 1/6 & 10 & 55 & 55 & 320 &  51.54\\
  GLEK & 1/6 & 9& 151 & 34 & 644 & 14.17 \\
  Algorithm~\ref{alg:Krylov}& 1/6 & 6 & 72 & 60 & 36  & 3.77\\
  \hline
  \hline
  BilADI (4 Wach.) & 1/5 & 14 & 71 & 71 & 588 & 55.15 \\
  BilADI (8 $\mathcal{H}_2$-opt.) & 1/5 & 14 & 69 & 69 & 586 & 54.31  \\
  GLEK & 1/5 & 12 & 173 & 39 & 1016& 22.06 \\
  Algorithm~\ref{alg:Krylov}& 1/5 & 6 & 72 & 61 & 36 & 4.23\\
  \hline
  \hline
  BilADI (4 Wach.) & 1/4 & 24 & 89 & 89 & 1454 & 67.61 \\
  BilADI (8 $\mathcal{H}_2$-opt.) & 1/4 & 23 & 89 & 89 & 1371& 66.83  \\
  GLEK & 1/4 & 21 & 218 & 50 & 2348& 51.49 \\
  Algorithm~\ref{alg:Krylov} & 1/4 & 8 & 96 & 81 & 48 & 6.72 \\
 \end{tabular}}
 \caption{MIMO example. Comparison of low-rank methods for $n=50000$.
 }\label{Ex.2_Tab.1}
\end{table}
We notice that, the number of linear solves that our projection method requires is always
much less than for the other methods.
Moreover, it seems that moderate variations of $\gamma$, that correspond to variations of $\rho(\mathcal{L}^{-1}\Pi)$, have a smaller influence on the number of iterations in our method compared to the other algorithms.

\subsection{A low-rank problem}\label{sec:low-rank}
We now consider the following generalized Lyapunov equation
\begin{equation}\label{eq.ex.1}
AX+XA^T+uv^TXvu^T=CC^T,
\end{equation}
where $A=n^2\mbox{tridiag}(1,-2,1)$ 
and $u,v,C\in\mathbb{R}^n$ are random vectors
with unit norm. 
We use Algorithm~\ref{alg:Krylov},
and as proposed in Section~\ref{sect:low_rank_Krylov}, we select $\bar C_1=\bar C_2=(C, \ u)$ as starting blocks. In Table~\ref{Ex.1_Tab.1} we report the results of the comparison to the other methods.
\begin{table}[!ht]
\centering
{\small
 \begin{tabular}{r|r|r|r|r|r|r}
  & $n$ & Its. & Mem. & rank($X$) & Lin. solves & CPU time \\
  \hline 
  \hline
  BilADI (4 Wach.) & 10000 &  60    &57 & 57 & 2462 & 4.25\\
  BilADI (8 $\mathcal{H}_2$-opt.) & 10000 & 42   &55 & 55 & 1420& 2.54 \\
  GLEK & 10000 & 4  &240 & 28 & 310& 3.10\\
  Algorithm~\ref{alg:Krylov} & 10000 & 46  &184 & 49 & 92 &1.87 \\
  \hline
  \hline
  BilADI (4 Wach.) & 50000 &  327    &61 & 61 & 18673 & 315.56 \\
  BilADI (8 $\mathcal{H}_2$-opt.) & 50000 & 96   &61 & 61 & 4580& 81.47 \\
  GLEK & 50000 & 4  &454 & 28 & 565& 24.78\\
  Algorithm~\ref{alg:Krylov} & 50000 & 78  &312 & 47 & 156 & 14.09 \\
  \hline
  \hline
  BilADI (4 Wach.) & 100000 &  -    &- & - & - & - \\
  BilADI (8 $\mathcal{H}_2$-opt.) & 100000 & 84   &65 & 65 & 4058& 174.04  \\
  GLEK & 100000 & 4  &457 & 29 & 631& 66.77\\
  Algorithm~\ref{alg:Krylov} & 100000 & 97  &388 & 44 & 194 & 37.00 \\
 \end{tabular}}
 \caption{Low-rank example. Comparison of low-rank methods varying $n$.
 }\label{Ex.1_Tab.1}
\end{table}
We notice that our approach requires the lowest number of linear solves. The ADI approaches demand the lowest storage because of the column compression strategy performed at each iteration. However, due to the large number of linear solves, these methods are slower compared to our approach. For large-scale problems the BilADI method with 4 Wachspress shifts does not converge in $500$ iterations. GLEK provides the solution with the smallest rank. 
If we replace the matrix $A$ with $A/n^2$ in equation \eqref{eq.ex.1}, neither BilADI nor GLEK converge since the Lyapunov operator is no longer dominant, i.e., $\rho(\mathcal{L}^{-1}\Pi) > 1$. However, our algorithm still converges and, for $n=10000$, it provides a solution $X$ in $46$ iterations with $\rank(X)=184$. In this case, the projected problems are solved by using the method presented in \cite[Section 3]{DammDirectADI} since the approach described in the Section~\ref{sec:small} cannot be used.

\subsection{Inhomogeneous Helmholtz equation} \label{sect:LaplaceChi}
In the last example, we analyse the complexity of Algorithm~\ref{alg:Krylov} for solving a large-scale generalized Sylvester equation stemming from a finite difference discretization of a PDE. More precisely, we consider the following inhomogeneous Helmholtz equation
\begin{equation}\label{Ex.3pde}
 \left\{\begin{array}{l}
         -\Delta u(x,y)+\kappa(x,y)u(x,y)=f(x,y), \quad (x,y) \in [0,1] \times \RR,\\
 u(x,0)= u(x,1)=0,\\
 u(x,y+1)=u(x,y).
        \end{array}\right.
\end{equation}
The boundary conditions are periodic in the $y$-direction and homogeneous-Dirichlet in the $x$-direction. The wavenumber 
$\kappa(x,y)$ and the forcing term $f(x,y)$ are 1-periodic functions in the $y$-direction. In particular they are respectively the periodic extensions of the scaled indicator functions $\chi_{[0,1/2]^2}$ and $100 \chi_{[1/4,1/2]^2}$.
The discretization of equation \eqref{Ex.3pde} with the finite difference method, using $n$ nodes multiple of $4$, leads to 
the following generalized Sylvester equation
\begin{equation}\label{Ex3.Lyap}
 AX+XB^T+NXN^T=CC^T,
\end{equation}
where
$B=-\mbox{tridiag}(1,-2,1)/h^2$,
$h=1/(n-1)$ is the mesh-size, $A=B-(e_1, e_n) (e_n, e_1)^T/h^2$, and
$$N=\begin{pmatrix}
O_{n/2} & O_{n/2} \\
O_{n/2} & I_{n/2}\\
   \end{pmatrix}\in\mathbb{R}^{n\times n},\;
  C=(c_1,\ldots,c_n)^T,\,
  c_i=
  \begin{cases}
  10, \;\mbox{if }i\in[n/4,n/2],\\
  0, \;\mbox{otherwise.}
  \end{cases}
$$
A direct computation shows that
$[A,N] =U \tilde U^T$
and 
$[B,N] =Q \tilde Q^T$
where
\begin{align*}
& U=n (e_{n/2+1}, e_{n/2}, e_1, e_n), 
&& \tilde U=n ( e_{n/2}, -e_{n/2+1},-e_n,e_1),	\\
& Q=n (e_{n/2+1}, e_{n/2}), 
&& \tilde Q=n(e_{n/2}, -e_{n/2+1}). 
\end{align*}
Algorithm~\ref{alg:Krylov} is not applicable to equation~\eqref{Ex3.Lyap} since the matrix $A$ is singular. 
However, in our approach it is possible to shift the Sylvester operator. In particular we can rewrite equation~\eqref{Ex3.Lyap} as 
\begin{align*}
 (A+I)X+XB^T+NXN^T-X=CC^T .
\end{align*}
It is now possible to apply Algorithm~\ref{alg:Krylov} since $A+I$ is nonsingular. 
For this problem it holds $N^2=N$ and then $\GK_{\ell}(N, I; C) = \vspan((C, NC))$ for all $\ell \geq 1$.
We now note that $[A+I,N]=[A,N]$, and that $NC=0$ and 
$\vspan((U,NU))=\vspan(U)$. Hence, according to Theorem~\ref{thm:Krylov} we select $\bar C_1  = (C, U)$ and $\bar C_2 = (C, Q)$ as starting blocks. Notice that, with this choice, Algorithm~\ref{alg:Krylov} provides an approximation of $X^{(\ell)}$ for every $\ell \geq 0$.
We fix the number of iterations $d=30$ in Algorithm~\ref{alg:Krylov}, and we vary the problem size $n$. In Figure~\ref{Ex3Fig1} we report the percentages of the overall execution time devoted to the orthogonalization procedure (Steps~\ref{step:genKa}-\ref{step:genKb}), to the solution of the inner problems (Step~\ref{step:proj_problem}) and to the remaining steps of the algorithm. We can see that for very large problems, most of the computational effort is dedicated to the orthogonalization procedure. See Figure~\ref{fig:conv} for an illustration of the converge history for the problem of size $n=10000$.

\begin{figure}%
    \centering
    \subfloat[Residual norm history for problem size $n=10000$.]{{
	\hspace{-1cm}
	\includegraphics{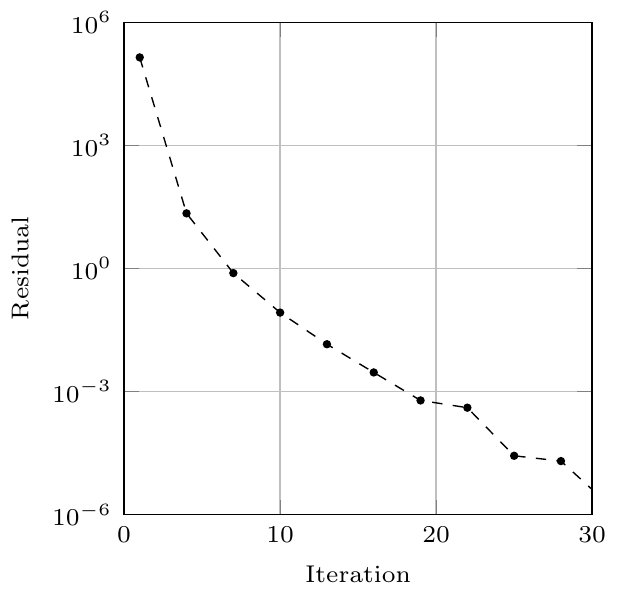} \label{fig:conv}
    }} \hspace{0.3cm}
    \subfloat[CPU time (percent) of the main parts of Algorithm~\ref{alg:Krylov} with $d~=~30$.]{{
 	\includegraphics{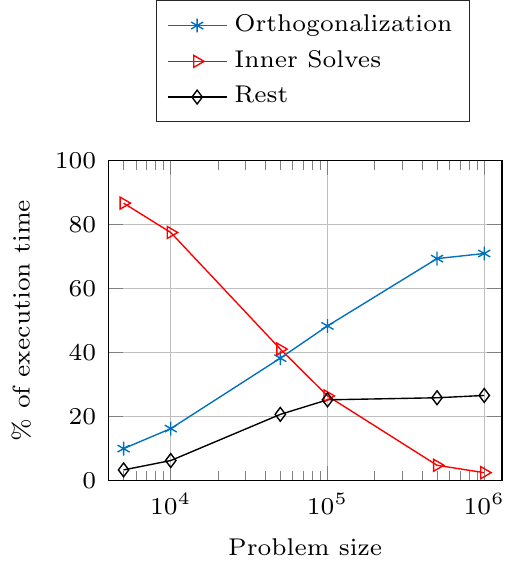}
 	\label{Ex3Fig1}
    }}%
    \caption{Simulations for the Inhomogeneous Helmholtz equation.}%
    \label{fig:example}%
\end{figure} 
 
\section{Conclusions and outlook}\label{sec:conc}
The method that we have proposed for solving~\eqref{eq:Sylv} is directly based on the low-rank commutation feature of the matrix coefficients \eqref{eq:com}. We have applied and adapted our procedure to problems in control theory and discretization of PDEs that naturally present this property. The structured matrices that present this feature are already analysed in literature although, to our knowledge, this was never exploited in the setting of Krylov-like methods for matrix equations. Low-rank commuting matrices are usually studied with the \emph{displacement operators}. More precisely, for a given matrix $Z$, the displacement operator is defined as $F(A):=AZ-ZA$. 
For many specific choices of the matrix $Z$, e.g., Jordan block, circulant, etc., it is possible to characterize the displacement operator and describe the matrices that are low-rank commuting with $Z$. See, e.g., \cite{kailath1995displacement,Beckermann2016}, \cite[Chap. 2, Sec. 11]{bini2012polynomial} and references therein. The theory concerning the displacement operator may potentially be used to classify the problems that can be solved with our approach.

The approach we have pursued in this paper is based on the extended Krylov subspace method. However, it seems to be possible to extend this to the rational Krylov subspace method \cite{Druskin.Simoncini.11} since, the commutator $[A,N]$ is invariant under translations of the matrix $A$. Further research is needed to characterize the spaces and study efficient shift-selection strategies.

In each iteration of Algorithm~\ref{alg:Krylov} the residual can be computed without explicitly constructing the current approximation of the solution but only using the solution of the projected problem. It may be possible to compute the residual norm even without explicitly solving the projected problems as proposed in \cite{Palitta2016a} for Lyapunov and Sylvester equations with symmetric matrix coefficients.

In conclusion, we wish to point out that the low-rank approximability characterization may be of use outside of the scope of projection methods. For instance, the Riemannian optimization methods are designed to compute the best rank $k$ approximation (in the sense of, e.g., \cite{kressner2016preconditioned,vandereycken2010riemannian}) to the solution of the matrix equation. This approach is effective only if $k$ is small, i.e., the solution is approximable by a low-rank matrix, for which we have provided sufficient conditions.

\section*{Acknowledgment}
We wish to thank Tobias Breiten (Graz University) for kindly providing the code 
which helped us to implement BilADI \cite{Benner:2013:low} used in Section~\ref{sec:num}. We also thank Stephen D. Shank 
(Temple University) for providing us with the GLEK code before its on-line publication.

This research commenced during a visit of the third author to the KTH Royal Institute of Technology.
The warm hospitality received is greatly appreciated. The work of the third author is partially supported 
by INdAM-GNCS under the 2017 Project ``Metodi numerici avanzati per equazioni e funzioni di matrici con struttura''. 
The other authors gratefully acknowledge the support of the Swedish Research Council under Grant No. 621-2013-4640.

\bibliographystyle{elsart-num-sort}
\bibliography{manuscript_bib}

\end{document}